\documentclass{icmart}
\usepackage{longtable}
\usepackage{graphicx}

\contact[vineet.kumar.cer11@iitbhu.ac.in]{Indian Institute of Technology\newline
Banaras Hindu University\newline
Varanasi, UP 221005
}

\newtheorem{theorem}{Theorem}[section]

\newtheorem{conjecture}[theorem]{Conjecture}

\newtheorem*{coro}{Corollary} 

\theoremstyle{definition}
\newtheorem{definition}[theorem]{Definition}

\title[Prime number generation and factor elimination]{Prime number generation and
\newline factor elimination}

\author[Vineet Kumar]
{Vineet Kumar\thanks{Thanks to Abhijit Phatak and Phani Ravi Teja of Indian Institute of Technology, Banaras Hindu University, Varanasi.}}

\begin{document}

\begin{abstract}
We have presented a multivariate polynomial function termed as factor elimination function,by which, we can generate prime numbers. This function's mapping behavior can explain the irregularities in the occurrence of prime numbers on the number line. Generally the different categories of prime numbers found till date, satisfy the form of this function. We present some absolute and probabilistic conditions for the primality of the number generated by this method. This function is capable of leading to highly efficient algorithms for generating prime numbers.
\end{abstract}



\maketitle

\section{Introduction}
\begin{definition}
Prime Numbers are those numbers which appear on the number line and are divisible by only 1 and the number itself. Hence these numbers have only two factors.
\end{definition}
In this paper we have presented a multivariate polynomial function termed as factor elimination function which is supposed to generate all prime numbers occurring on the number line. We call it as factor elimination because it generates a number by reducing the divisibility by most of the prime factors, we can generate small or big numbers from the function depending upon the factors and certain values taken under consideration. For the generated number some absolute conditions for primality are given. Probabilistic conditions explain why the image of the function cannot be prime, or could be a prime under particular probability conditions. The reason behind the various categories of prime numbers is also explained by this function. There are two cases, one is assured prime number generation where there is no need to pass primality test. While the second case requires a primality test to be passed and hence there is some definite probability associated.
\section{Generalized Proof of Euclid's Theorem}
\begin{theorem}For any finite set of prime numbers, there exists a prime number not in that set.
\end{theorem}

\begin{coro}
There are infinitely many prime numbers. 
\end{coro}

\begin{coro}
There is no largest prime number. 
\end{coro}

\begin{proof}
Let us assume a set S consisting of prime numbers which is  partitioned into two distinct sets of prime numbers A and B.
\begin{align}
A& = \left\{ {P_1, P_2, P_3,...,P_{i-1},P_i,P_{i+1},...,P_{n-1},P_n}\right\}
\\
B& = \left\{ {O_1, O_2, O_3,...,O_{j-1},O_j,O_{j+1},...,O_{m-1},O_m}\right\}
\end{align}
Where $S=A\cup$B, $A\cap$B = $\emptyset$ and n and m are some positive integers. Consider the following mathematical operation defined as R
\begin{align}
R& = (P_1\times P_2\times ...\times P_i\times ... \times P_n) \pm (O_1\times O_2\times ...\times O_j\times ...\times O_m)
\end{align}
Let us now assume that $P_i$ is a prime factor of R. Let W = $R\div P_i$. Thus clearly W should be an integer,
\begin{align}
W&= [(P_1\times ...\times P_i\times ... \times P_n) \pm (O_1\times ...\times O_j\times ...\times O_m)]\div P_i
\\
W&= (P_1\times ...\times P_i\times ...\times P_n) \pm (O_1\times  ...\times O_j\times ...\times O_m)\div P_i
\end{align}
But the term, $(O_1\times  ...\times O_j\times ...\times O_m)\div P_i$ can never be a whole number because, $P_i$ does not belong to set B, and a prime number cannot be factor of any other prime number. Hence by contradiction, it is proved that $P_i$ can never be a prime factor of R.

Similarly, Let us consider each Prime number $P_i$ to be raised to the power $a_i$ and each $O_j$ raised by power $b_j$. Then the resultant:
\begin{align}
R& = (P^{a_1}_1\times ...\times P^{a_i}_i\times ... \times P^{a_n}_n) \pm (O^{a_1}_1\times  ...\times O^{a_j}_j\times ...\times O^{a_m}_m)
\end{align}
Again, let us assume that $P_i$ is a prime factor of R then let W = $R\div P_i$. Thus W should be an integer which means,
\begin{align}
W&= [(P^{a_1}_1\times ...\times P^{a_i}_i\times ... \times P^{a_n}_n) \pm (O^{a_1}_1\times  ...\times O^{a_j}_j\times ...\times O^{a_m}_m)]\div P_i
\\
W&= (P^{a_1}_1\times  ...\times P^{a_i-1}_i\times ...\times P^{a_n}_n) \pm (O^{a_1}_1\times  ...\times O^{a_j}_j\times ...\times O^{a_m}_m)\div P_i
\end{align}
The term, $(O^{a_1}_1\times  ...\times O^{a_j}_j\times ...\times O^{a_m}_m)\div P_i$ can never be a whole number because, $P_i$ does not belong to set B. Thus we similarly have proved, by contradiction, that $P_i$ can never be a prime factor of R.
This proof of contradiction shows that at least one additional prime number exists that doesn't belong to set S.
\end{proof}

The important thing to consider here is that if in some case A or B are empty sets,  then 1 must be considered as the only element of that set for e.g A=$\left\{2,3\right\}$ and B=$\left\{1\right\}$.
\section{Theory of Factor Elimination}
For two distinct sets A and B consisting of prime numbers, let P be the largest prime number in either of the sets.
\begin{align}
A\cap B &= \emptyset
\\
A\cup B &= S
\end{align}
Let $x_i\in A$ and $y_i\in B$; then for
\begin{align}
R&= \mid\Pi x^{a_i}_i \pm \Pi y^{b_j}_j\mid
\end{align}
Let us name this function as the Factor Elimination function. We generate a pair of resultants; $R^+$ by addition and $R^-$ by subtraction. The probability that R is prime is very high and it must be prime if,
\begin{align}
\sqrt{R}&\le P
\end{align}
But practically, we realize that it is very difficult to verify $\sqrt{R}\le$ P, and highly efficient algorithms are required. In that case, we can depend on the probability that, R is most likely a prime number, and we can determine it by the primality test like the Rabin-Miller Probabilistic Primality Test [2].  It is well known that if a prime factor of R exists other than R itself, then at least one of those prime factors must be less than $\sqrt{R}$ [3].
\section{Probability for being a Prime Number}
Any prime number that is less than or equal to P, cannot be a factor of R. If $\sqrt{R}\le$ P and C = $\sqrt{R}$, then let the number of prime numbers which may be prime factors of R lying in between P and C be denoted by N. For this, we can use the prime counting function[4] and hence the number of primes capable of dividing R is given by
\begin{align}
N=\pi(C)- \pi(P)
\end{align}
Where N represents the exact number of primes that exist between P and C. Instead, we can also use rough approximation by Prime Number Theorem[5] for calculating N represented by symbol $N^{\circ}$.
\begin{align}
N^{\circ}  = \frac{C}{ln(C)}-\frac{P}{ln(P)}
\end{align}
As the value of R increases, the value of C also increases correspondingly, and the gap between P and C widens on the number line as a result of which the value of N also increases.
This implies that the probability that a number could be a factor of R increases. So, it was concluded that, the closer R is to $P^2$, the probability of R being a prime number increases.
The total number of primes not greater than C which is $\sqrt{R}$ is equal to $\pi(C)$. Thus N represents the primes which can be possible factors of R. We can conclude from here that, the probability for event X where X is defined as divisibility of R by set of primes comprising of N elements.
\begin{align}
P(X)=\frac{N}{\pi(C)}= \frac{\pi(C)- \pi(P)}{\pi(C)}= 1- \frac{\pi(P)}{\pi(C)}
\end{align}
\begin{align}
1-P(X)&= \frac{\pi(P)}{\pi(C)}
\end{align}

Where 1-P(X) represent the probablity of R to be a prime and for the case $P\ll$C this probability tends to zero.
Let us suppose, we do not choose some prime between 1 and P. Let T be the set of such prime numbers. Consider R(P) as the residual prime function which counts the number of primes in T. Now the above equation (16) can be written as following
\begin{align}
1-P(X)&= \frac{\pi(P)-R(P)}{\pi(C)}
\end{align}
For e.g
\begin{align}
A&=\left\{2,3\right\}, B=\left\{1\right\}
\\
R&= 2^2 \times 3^2-1 = 36 -1 = 35
\\
C&\approx 5.916, \pi(C)=3, \pi(P)=2, R(P)=1
\end{align}
 Let 5 be an element of the set  T. We must not consider those values of R where the difference or addition of last digits from both the multiplied result sets A and B is divisible by 5. We add an exception to accommodate R=5. For instance, in the example above 6 and 1 are last digits for 36 and 1 respectively. 
 \\The consideration of the value of R(P) is very important when  we talk of prime number generation algorithms and it should be minimum. Additionally the set T should consist of larger prime numbers only.
\section{Twin Prime Conjecture}
\begin{definition}
A twin prime is a prime number that differs from another prime number by two.
\end{definition}
\begin{conjecture}
There Are Infinitely Many Prime Twins.
\end{conjecture}
Considering large pair of twin primes, we can say that the probability of getting a twin prime is square of the probability that we get a prime at $P\ll$C. 
\begin{align}
R&= \mid\Pi x^{a_i}_i \pm 1\mid
\end{align}
The important thing to notice here is that the probability of getting a pair of resultant as prime is equal in this case. 
\section{Classification of Prime Numbers}
This factor elimination function is also useful for giving a general mathematical form for most of the various classifications of prime numbers given till date. This is presented in a tabular form given below.
\begin{center}
\begin{longtable}{|p{2cm}|p{2.6cm}|p{2.3cm}|p{4cm}|}
\caption[Categorization of Prime Numbers]{Categorization of Prime Numbers} \label{grid_mlmmh} \\

\hline \multicolumn{1}{|c|}{\textbf{\parbox{2cm}{Prime\\Category}}} & \multicolumn{1}{c|}{\textbf{Initial Form}} &
\multicolumn{1}{c|}{\textbf{Simplified Form}} & \multicolumn{1}{c|}{\textbf{\parbox{4cm}{Factor Elimination\\Criteria}}} \\ \hline 
\endfirsthead

\multicolumn{3}{c}%
{{\bfseries \tablename\ \thetable{} -- continued from previous page}} \\
\hline\multicolumn{1}{|c|}{\textbf{\parbox{2cm}{Prime\\Category}}} & \multicolumn{1}{c|}{\textbf{Initial Form}} &
\multicolumn{1}{c|}{\textbf{Simplified Form}} & \multicolumn{1}{c|}{\textbf{\parbox{4cm}{Factor Elimination\\Criteria}}} \\ \hline  
\endhead

\hline \multicolumn{4}{|r|}{{Continued on next page}} \\ \hline
\endfoot

\hline \hline
\endlastfoot

Carol primes               & $(2^n$-$ 1)^2$ -– 2    & $(2^n$-$ 1)^2$ -– 2                 &  $R^+$, A = $\left\{2^n-1\right\}$, B=$\left\{2\right\}$, $a_1$= 2                          \\ \hline
Centered decagonal primes               &     $5(n^2$-$   n)$ + 1       &       5n(n - 1) + 1           &   $R^+$, A = $\left\{5,n,n-1\right\}$, B=$\left\{1\right\}$                             \\ \hline
       Centered heptagonal primes        &        $\frac{(7n^2 -  7n + 2)}{2}$     &   7n(n-1)$\div$ 2 + 1              &$R^+$,  A=\{7,n/2,n-1\} if n is even, A=\{7,n,(n-1)/2\}  if n is odd, B=\{1\}                         \\ \hline
          Centered square primes     & $n^2$ + $(n+1)^2$             & 2n(n+1)+1                & $R^+$, A=\{2,n,n+1\}, B=\{1\}                            \\ \hline
        Centered triangular primes       &   $\frac{(3n^2 + 3n + 2)}{2}$           & 3n(n+1)/2 + 1                &  $R^+$, A=\{3,n/2,n+1\} if n is even else A=\{3,n,(n+1)/2\}  if n is odd                           \\ \hline
       \parbox{2cm}{Cuban primes\\(Case I)}       &    $\frac{m^3-n^3}{m-n}$, m = n+1          &     3n(n+1)+1            &    $R^+$, A=\{3,n,n+1\}, B=\{1\}                      \\ \hline
       \parbox{2cm}{Cuban primes\\(Case II)}     &       $\frac{m^3-n^3}{m-n}$, m = n+2       & 3n(n+2)+$2^2$                & $R^+$, A=\{3,n,n+2\}, B=\{2\}, $b_1$=2                      \\ \hline
         Cullen primes      &      $n\times 2^n$ + 1        &         $n\times 2^n$ + 1          &    $R^+$, A=\{2,n\}, B=\{1\}, $a_1$=n                             \\ \hline
        Double factorial primes       &     $ n!! \pm 1 $       &      $\Pi x^{a_i}_i \pm 1$         &      R, A=S, B=\{1\}, $a_i\geq 2$ $\forall$ $P_i$ $\in$ S             \\ \hline
          Double Mersenne primes     &     $2^{2^{p -− 1}}$ -– 1         &       $2^{2^{p -− 1}}$ – 1               &   $R^-$,      A=\{2\}, B=\{1\}, $a_1= 2^{p-1}$ where p is some Prime                      \\ \hline
         Eisenstein primes without imaginary part      &    3n-1          &       3n-1          &      $R^-$,     A=\{3,n\}, B=\{1\}                   \\ \hline
         Euclid primes      &       $p_n \texttt{\#} $ + 1       &      $\Pi p_i  $ + 1           &    $R^+$, A=S, B=\{1\}                        \\ \hline
          Factorial primes     &     $ n! \pm 1 $          &   $\Pi x^{a_i}_i \pm 1$                &   R, A=S, B=\{1\}, $a_i\geq 2$ $\forall$ $P_i$ $\in$ S                             \\ \hline
         Fermat primes      &       $2^{2^n}$ -– 1         &             $2^{2^n}$ -– 1      &        $R^-$,      A=\{2\}, B=\{1\}, $a_1= 2^n$, where n is Positive Integer                      \\ \hline
       Fibonacci primes        &          $p_n \texttt{\#} $ + $P_m$        &      $\Pi p_i  $ + $P_m$             &        $R^+$, A=S-\{$P_m$\}, B=\{$P_m$\}, Where $P_m$ is maximum prime                  \\ \hline
       Gaussian primes        &     4n+3         &  $2^2$n+3               &   $R^+$,  A=\{2,n\}, B=\{3\}, $a_1$=2                        \\ \hline
           Generalized Fermat primes base 10    &    $10^n$ + 1          &           $2^n 5^n$+1      &         $R^+$, A=\{2,5\},B=\{1\}, $a_1=a_2=n$                     \\ \hline
         Kynea primes      &        $(2^n + 1)^2$ -− 2      &      $(2^n + 1)^2$ -− 2            & $R^-$, A = \{$2^n$+1\} , B=\{2\}, $a_1$ = 2                  \\ \hline
      Leyland primes         &      $m^n +n^m$        &       $m^n +n^m$           &  $R^+$, A=\{n\}, B=\{m\}, $a_i=m$, $b_1=n$                     \\ \hline
         Mersenne primes      &     $2^p -− 1$         &    $2^p -– 1$             & $R^-$, A=\{2\},B=\{1\}, $a_1=n$, Where p is some prime                     \\ \hline
        Odd primes       &         2n − 1     &        2n − 1         &   $R^-$, A=\{2,n\}, B=\{1\}                   \\ \hline
        Palindromic wing primes       &   $\frac{a(10^m-1)}{9}\pm b*10^\frac{m}{2}$           &     $\frac{a(10^m-1)}{9}\pm b*10^\frac{m}{2}$            & R, A=\{ $\frac{a(10^m-1)}{9}$\} B=\{b,$10^\frac{m}{2}$\}                      \\ \hline
       Pierpont primes        &      $2^u3^v$ + 1        &       $2^u3^v$ + 1          & R, A=\{2,3\}, B=\{1\}, $a_1$=u, $a_2=v$                   \\ \hline
       Primes of the form n4 + 1        &       $n^4$ + 1       &    $n^4$ + 1              & $R^+$, A=\{n\}, B=\{1\}, $a_1$=4                         \\ \hline
       Primorial primes        &      $p_n \texttt{\#}  \pm 1 $        &   $ \Pi p_i \pm 1$          &      R, A=S, B={1}                       \\ \hline
         Proth primes      &  $k \times 2n + 1$, With odd k and k < 2n            &       $k \times 2n + 1$          & $R^+$,        A={2,k,n}, B={1} , With odd k and k < 2n                    \\ \hline
      Pythagorean primes         &       4n + 1       &   $2^2n$+1              & $R^+$, A=\{2,n\}, B=\{1\}                      \\ \hline
       Quartan primes        &  $x^4 + y^4$            &   $x^4 + y^4$              & $R^+$, A=\{x\}, B=\{y\}, $a_1= b_1=4$                      \\ \hline
         Solinas primes      &     $2^a \pm 2^b \pm 1$         &    $2^a \pm 2^b \pm 1$             &   R, A=\{$2^a$\}, B=\{$2^b \pm 1$\}                   \\ \hline
        Star primes       &     6n(n -− 1) + 1         &  $2 \times 3 \times n (n −- 1)$ + 1                & $R^+$, A=\{2,3,n,n-1\}, B=\{1\}                         \\ \hline
       Thabit number primes        &  $3\times 2n$ − 1            &       $3\times 2n$ -− 1           &  $R^-$, A=\{2,3\}, B=\{1\}, $a_1=n$                      \\ \hline
         Woodall primes      &   $n\times 2^n – 1$           &   $n\times 2^n – 1$              &   $R^-$, A=\{2,n\}, B=\{1\}, $a_1=n$                          \\ \hline

\end{longtable}
\end{center}
\begin{table}[h]
\caption{Categorization which depends on occurrence of more than one prime.}
\centering
\begin{tabular}{|p{3cm}|p{4cm}|p{4cm}|}
\hline
Property & Meaning &  Factor Elimination Criteria\\ \hline
    Twin primes     &   (p, p+2) are both prime.      & $R= \mid\Pi x^{a_i}_i \pm 1\mid$  \\ \hline
      Sexy primes   &   (p, p+6) are both prime      &  $R= \mid\Pi x^{a_i}_i \pm 3\mid$\\ \hline
     Sophie Germain prime    &     p and 2p+1 are both prime    & R, A=\{2,p\}, B=\{1\} \\ \hline
    Safe prime     &   \parbox{4cm}{p and (p−-1)/2 are both prime\\ Consider p = 2k+1}      & R, A=\{2,k\}, B=\{1\} \\ \hline
     Prime triplets   &  \parbox{4cm}{(p, p+2, p+6) or (p, p+4, p+6) are all prime}      &  \parbox{4cm}{$R= \mid\Pi x^{a_i}_i \pm y\mid$, y =1,3,5\\Where at least one of the R is having divisibility  by 3}\\ \hline
     Prime quadruplets    &  (p, p+2, p+6, p+8) are all prime       &  \parbox{4cm}{$R= \mid\Pi x^{a_i}_i \pm y\mid$, y =1,3,5\\Where at least one of the R is having divisibility  by 3} \\ \hline
     Primes in residue classes    & an + d for fixed a and d  & $R^+$, A=\{a,n\}, B=\{d\} \\ \hline
\end{tabular}
\end{table}
\section{Probability Comparison of a Simple Algorithm }
Let us now compare the usefulness of this function. With the of simplest algorithms used for generating prime numbers for public key encryption. Considering the Factor Elimination Function,
\begin{align}
R&= \mid\Pi x_i \pm \Pi y_j\mid
\end{align}

$a_i$ and $b_j$ had been considered unity and an algorithm had been created, that can be briefly described as:
\\
\\
1.	Two random lists were created with the help of all possible combinations of elements that are in the sets A and B.
\\
2.	All the elements of A were multiplied. Similarly all elements of B were multiplied. These were saved as two separate results.
\\
3.	The two results were either subtracted or added, and the absolute values were taken.
\\
4.	The final resultants were verified by the Primality Test.
\\
\\
The probability of generating an n-bit prime with this method was high as compared to earlier methods. In the previous implementation the probability that a number is prime is $\frac{1}{ln(2^n)}$ [7]. All possible combinations for a consecutive list of prime numbers of size L were analyzed. As the value of L was increased, the total number of combinations and the total number of prime numbers thus generated was calculated.
The data obtained is tabulated below:

\begin{table}[h]
  \centering
  \caption{Primes Distribution}
    \begin{tabular}{|p{3cm}|p{3cm}|p{2.5cm}|p{2.5cm}|}
    \hline
    Total Primes (P) & Combinations (C) & $\log_2(C)$    & Ratio ($\frac{P}{C}$) \\ \hline
    1     & 1     & 1     & 1 \\ \hline
    7     & 4     & 2     & 1.75 \\ \hline
    25    & 16    & 4     & 1.563 \\ \hline
    79    & 64    & 6     & 1.234 \\ \hline
    256   & 256   & 8     & 1 \\ \hline
    887   & 1024  & 10    & 0.866 \\ \hline
    2808  & 4096  & 12    & 0.686 \\ \hline
    10405 & 16384 & 14    & 0.635 \\ \hline
    34450 & 65536 & 16    & 0.526 \\ \hline
    120504 & 262144 & 18    & 0.46 \\ \hline
    418223 & 1048576 & 20    & 0.399 \\ \hline
    1597836 & 4194304 & 22    & 0.381 \\ \hline
    5926266 & $1.70\times 10^7 $ & 24    & 0.353 \\ \hline
    $2.10\times10^7$ & $6.70\times 10^7$ & 26    & 0.318 \\ \hline
    $7.70 \times 10^7$ & $2.70 \times 10^8$ & 28    & 0.288 \\ \hline
    \end{tabular}%
  \label{tab:addlabel}%
\end{table}%
For instance, consider a prime list of length L=25. Then it can generate a prime number which is in the range of 42 to 119 bit length. The probability for generating that minimum bit length of prime number can be supposed to be $\frac{1}{ln(2^{42})}$ which is approximately equal to 0.0121 where the obtained probability was 0.3532.
\begin{table}[htbp]
  \centering
  \caption{Probability Comparison}
    \begin{tabular}{|r|r|r|r|r|r|}
    \hline
    List Length (L)    & Max Bit (g)    & Min Bit (h)    & $\frac{1}{ln(2^g)}$ & $\frac{1}{ln(2^h)}$ & Probablity ($\frac{P}{C}$) \\ \hline
    11    & 38    & 11    & 0.038 & 0.1312 & 0.8662 \\ \hline
    13    & 49    & 17    & 0.0294 & 0.0849 & 0.6855 \\ \hline
    15    & 59    & 21    & 0.0245 & 0.0687 & 0.6351 \\ \hline
    17    & 70    & 22    & 0.0206 & 0.0656 & 0.5257 \\ \hline
    19    & 81    & 31    & 0.0178 & 0.0465 & 0.4597 \\ \hline
    21    & 94    & 32    & 0.0153 & 0.0451 & 0.3988 \\ \hline
    23    & 104   & 40    & 0.0139 & 0.0361 & 0.381 \\ \hline
    25    & 119   & 42    & 0.0121 & 0.0343 & 0.3532 \\ \hline
    27    & 126   & 52    & 0.0114 & 0.0277 & 0.3178 \\ \hline
    29    & 143   & 55    & 0.0101 & 0.0262 & 0.2877 \\ \hline
    \end{tabular}%
  \label{tab:addlabel}%
\end{table}%
\begin{figure}
\centering 
\includegraphics[width=0.7\linewidth]{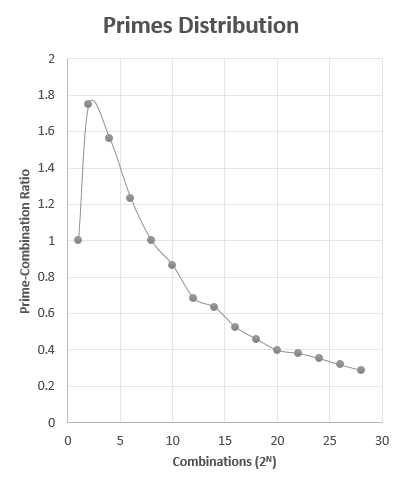}
\caption{Prime Probability Distribution }
\label{fig:graph}
\end{figure}
\\
\\
\\
\\

\section{Conclusion}
Factor elimination function is capable of generating all the prime numbers as well as can be used as a powerful tool for developing highly efficient prime numbers generating algorithms. This function is a multivariate polynomial function. Because every integer on number line can be represented as the sum or difference of two integers and these integers can be written in the form of multiplication of their factors. These generated numbers do not follow a regular pattern or a sequence under the given probabilistic condition for being a prime. With the help of prime counting function, we can explain this finite probability. Here, we have explained that most of the categorization of prime numbers discovered till now, are in some form following factor elimination function.
\section{Future Scope}
As we have demonstrated the application of factor elimination function in generating large prime numbers for encryption. A lot of further research in number theory and prime numbers can be done with the help of this method.
\section{References}

[1] James Williamson (translator and commentator), The Elements of Euclid, With Dissertations, Clarendon Press, Oxford, 1782, page 63.
 \newline[2] Rabin, Michael O. (1980), "Probabilistic algorithm for testing primality", Journal of Number Theory 12 (1): 128–138, doi:10.1016/0022-314X(80)90084-0
\newline[3] Crandall, Richard; Pomerance, Carl (2005), Prime Numbers: A Computational Perspective (2nd ed.), Berlin, New York: Springer-Verlag, ISBN 978-0-387-25282-7
\newline[4] Bach, Eric; Shallit, Jeffrey (1996). Algorithmic Number Theory. MIT Press. volume 1 page 234 section 8.8. ISBN 0-262-02405-5.
\newline[5] N. Costa Pereira (August–September 1985). "A Short Proof of Chebyshev's Theorem". American Mathematical Monthly 92 (7): 494–495. doi:10.2307/2322510. JSTOR 2322510.
\newline[6] Arenstorf, R. F. "There Are Infinitely Many Prime Twins." 26 May 2004. http://arxiv.org/abs/math.NT/0405509.
\newline
[7] Hoffman, Paul (1998). The Man Who Loved Only Numbers. Hyperion. p. 227. ISBN 0-7868-8406-1.5


\end{document}